\documentclass[12pt]{amsart}
\usepackage{amsfonts}
\usepackage{graphicx}
\usepackage{amscd}

\newtheorem{theorem}{Theorem}
\theoremstyle{plain}

\newtheorem{corollary}{Corollary}

\newtheorem{definition}{Definition}

\newtheorem{lemma}{Lemma}

\newtheorem{proposition}{Proposition}
\newtheorem*{remark}{Remarks}

\numberwithin{equation}{section}

\input{tcilatex}

\begin{document}
\title[ Fixed points of RSC mappings]{Weak and strong convergence theorems
for generalized nonexpansive mappings}
\author{Safeer Hussain Khan}
\address{Safeer Hussain Khan, Department of Mathematics, Statistics and
Physics, Qatar University, Doha 2713, State of Qatar.}
\email{safeerhussain5@yahoo.com; safeer@qu.edu.qa}
\author{Ibrahim Karahan}
\address{Department of Mathematics, Faculty of Science, Erzurum Technical
University, Erzurum, 25700, Turkey.}
\email{ibrahimkarahan@erzurum.edu.tr}
\date{%
\today%
}

\begin{abstract}
We consider a class of generalized nonexpansive mappings introduced by
Karapinar \cite{karapinar} and seen as a generalization of Suzuki
(C)-condition. We prove some weak and strong convergence theorems for
approximating fixed points of such mappings under suitable conditions in
uniformly convex Banach spaces. Our results generalize those of Khan and
Suzuki \cite{khansuz} to the case of this kind of mappings and, in turn, are
related to a famous convergence theorem of Reich \cite{reich} on
nonexpansive mappings.
\end{abstract}

\maketitle

\section{Introduction}

Let $E$ be a Banach space and let $K$ be a nonempty subset of $E$. A mapping 
$T$ on $K$ is called nonexpansive if $\left\Vert Tx-Ty\right\Vert \leq
\left\Vert x-y\right\Vert $ for all $x,y\in K$. The set of fixed points of $%
T $ is denoted by $F\left( T\right) ,$ i.e., $F\left( T\right) =\left\{ x\in
K:Tx=x\right\} .$ It is well known that if $E$ is uniformly convex and $K$
is a bounded, closed and convex subset of $E$, then $F\left( T\right) $ is
nonempty. Recall that a mapping $T:K\rightarrow K$ is called
quasi-nonexpansive \cite{diaz} if $\left\Vert Tx-p\right\Vert \leq
\left\Vert x-p\right\Vert $ for all $x\in K$ and $p\in F\left( T\right) $.
It is easy to see that every nonexpansive mappings with a fixed point is
quasi-nonexpansive mapping.

In 2008, Suzuki \cite{suzuki} introduced the concept of generalized
nonexpansive mappings (Condition (C)). This concept is weaker than
nonexpansiveness and stronger than quasi-nonexpansiveness.

\textbf{Condition (C) }For a mapping $T$ defined from a subset $K$ of a
Banach space $E$ into itself, $T$ is said to satisfy the condition (C) if%
\begin{equation*}
\frac{1}{2}\left\Vert x-Tx\right\Vert \leq \left\Vert x-y\right\Vert \text{
implies }\left\Vert Tx-Ty\right\Vert \leq \left\Vert x-y\right\Vert
\end{equation*}%
for all $x,y\in K.$

Suzuki \cite{suzuki} proved the following theorems for mappings satisfying
this condition:

\begin{theorem}
\label{SUZ1}\cite{suzuki} Let $T$ be a mapping on a compact convex subset $K$
of a Banach space $E$. Assume that $T$ satisfies condition (C). Define a
sequence $\left\{ x_{n}\right\} $ in $K$ by $x_{1}\in K$ and%
\begin{equation*}
x_{n+1}=\alpha Tx_{n}+\left( 1-\alpha \right) x_{n},
\end{equation*}%
where $\alpha $ is a real number belonging to $\left[ 1/2,1\right) $. Then $%
\left\{ x_{n}\right\} $ converges strongly to a fixed point of $T$.
\end{theorem}

\begin{theorem}
\label{SUZ2}\cite{suzuki} Let $T$ be a mapping on a weakly compact convex
subset $K$ of a Banach space $E$ with the Opial property. Assume that $T$
satisfies condition (C). Define a sequence $\left\{ x_{n}\right\} $ in $K$
by $x_{1}\in K$ and%
\begin{equation*}
x_{n+1}=\alpha Tx_{n}+\left( 1-\alpha \right) x_{n},
\end{equation*}%
where $\alpha $ is a real number belonging to $\left[ 1/2,1\right) $. Then $%
\left\{ x_{n}\right\} $ converges weakly to a fixed point of $T$.
\end{theorem}

It follows from Theorem \ref{SUZ1} and Theorem \ref{SUZ2} that for the
mapping $T$ satisfying the condition (C) defined on a convex subset $C$ of a
Banach space $E$, if either $C$ is compact or $C$ is a weakly compact and $E$
has the Opial property, then $T$ has a fixed point. (see \cite[Theorem 4]%
{suzuki})

In 2013, Khan and Suzuki \cite{khansuz} gave the following weak convergence
theorem for the mappings satisfying condition (C) defined on a bounded
closed convex subset $K$ of a uniformly convex Banach space $E$:

\begin{theorem}
\cite{khansuz} Let $E$ be a uniformly convex Banach space whose dual has the
Kadec--Klee property. Let $T$ be a mapping on a bounded, closed and convex
subset $K$ of $E$. Assume that $T$ satisfies condition (C). Define a
sequence $\left\{ x_{n}\right\} $ in $K$ by $x_{1}\in K$ and%
\begin{equation*}
x_{n+1}=\alpha Tx_{n}+\left( 1-\alpha \right) x_{n},
\end{equation*}%
where $\alpha $ is a real number belonging to $\left[ 1/2,1\right) $. Then $%
\left\{ x_{n}\right\} $ converges weakly to a fixed point of $T$ .
\end{theorem}

In 2013, Karapinar \cite{karapinar} introduced a new class of mappings as a
modification of mappings satisfying the condition (C) of Suzuki \cite{suzuki}%
.

\begin{definition}
\cite{karapinar} Let $T$ be a mapping on a subset $K$ of a Banach space $E$.
Then $T$ is said to satisfy Reich-Suzuki-(C) condition (in short,
(RSC)-condition) if%
\begin{equation*}
\frac{1}{2}\left\Vert x-Tx\right\Vert \leq \left\Vert x-y\right\Vert
\end{equation*}%
implies that%
\begin{equation*}
\left\Vert Tx-Ty\right\Vert \leq \frac{1}{3}\left( \left\Vert x-y\right\Vert
+\left\Vert y-Ty\right\Vert +\left\Vert x-Tx\right\Vert \right)
\end{equation*}%
for all $x,y\in K.$
\end{definition}

Also, Karapinar gave some properties of this kind of mappings and proved
some weak and strong convergence theorems for the mappings satisfying the
(RSC)-condition in Banach spaces. We will need the following which is
Proposition 1 of Karapinar \cite{karapinar}.

\begin{proposition}
\label{krap6} If a mapping $T$ satisfies (RSC)-condition and has a fixed
point, then it is quasi-nonexpansive mapping.
\end{proposition}

In this paper, we prove some weak and strong convergence theorems in
uniformly convex Banach spaces. For weak convergence theorem, we assume that
dual of uniformly convex Banach space has the Kadec-Klee property.

\section{\protect\bigskip Preliminaries}

Throughout this paper, we assume that all Banach spaces are real and denote
by $%
\mathbb{N}
$ the set of all positive integers unless stated otherwise. In this section,
we give some definitions, propositions and lemmas to use in the next section.

\begin{definition}
\cite{clark} Let $E$ be a Banach space. $E$ is called uniformly convex if
for each $\varepsilon >0$, there exists $\delta >0$ such that $\left\Vert
x+y\right\Vert \leq 2-\delta $ for all $x,y\in E$ with $\left\Vert
x\right\Vert =\left\Vert y\right\Vert =1$ and $\left\Vert x-y\right\Vert
\geq \varepsilon $.
\end{definition}

Uniformly convex spaces are common examples of reflexive Banach spaces. The
concept of uniform convexity was first introduced by Clarkson \cite{clark}
in 1936. In respect of these spaces, the following lemma is well known.

\begin{lemma}
\label{KIRK}\cite{kirk} Let $E$ be a uniformly convex Banach space. Let $%
\{x_{n}\}$ and $\{y_{n}\}$ be sequences in $E$ satisfying $%
\lim_{n\rightarrow \infty }\left\Vert x_{n}\right\Vert =1$, $%
\lim_{n\rightarrow \infty }\left\Vert y_{n}\right\Vert =1$ and $\
\lim_{n\rightarrow \infty }\left\Vert x_{n}+y_{n}\right\Vert =2.$ Then $%
\lim_{n\rightarrow \infty }\left\Vert x_{n}-y_{n}\right\Vert =0.$
\end{lemma}

The following lemma was proved in \cite{khansuz} utilizing Lemma \ref{KIRK}.
We shall use it in the proof of our main theorem.

\begin{lemma}
\label{KS}\cite{khansuz} Let $E$ be a uniformly convex Banach space and let $%
\{u_{n}\},\{v_{n}\}$ and $\{w_{n}\}$ be sequences in $E$. Let $d$ and $t$ be
real numbers with $d\in (0,\infty )$ and $t\in (0,1)$. Assume that $%
\lim_{n\rightarrow \infty }\left\Vert u_{n}-v_{n}\right\Vert =d$, $%
\limsup_{n\rightarrow \infty }\left\Vert u_{n}-w_{n}\right\Vert \leq (1-t)d$
and $\limsup_{n\rightarrow \infty }\left\Vert v_{n}-w_{n}\right\Vert \leq $ $%
td$. Then%
\begin{equation*}
\lim_{n\rightarrow \infty }\left\Vert tu_{n}+\left( 1-t\right)
v_{n}-w_{n}\right\Vert =0.
\end{equation*}
\end{lemma}

A Banach space $E$ is said to have the Kadec--Klee property if, for every
sequence $\left\{ x_{n}\right\} $ in $E$ which converges weakly to a point $%
x\in E$ with $\left\Vert x_{n}\right\Vert $ converging to $\left\Vert
x\right\Vert $, $\left\{ x_{n}\right\} $ converges strongly to $x.$ It is
known that uniformly convex Banach spaces have Kadec-Klee property \cite{aos}%
.

\begin{lemma}
\label{FK}\cite{fkr, kac} Let $E$ be a reflexive Banach space whose dual has
the Kadec--Klee property. Let $\left\{ x_{n}\right\} $ be a bounded sequence
in $E$ and let $y,z\in E$ be weak subsequential limits of $\left\{
x_{n}\right\} .$ Assume that for every $t\in \lbrack 0,1],$ $%
\lim_{n\rightarrow \infty }\left\Vert tx_{n}+\left( 1-t\right)
p-q\right\Vert $exists. Then $y=z$.
\end{lemma}

\begin{proposition}
\label{K} Let $T$ be a mapping on a subset $K$ of a Banach space $E$ and
satisfy (RSC)-condition. Then

\begin{description}
\item[(i)] $\left\Vert x-Ty\right\Vert \leq 7\left\Vert x-Tx\right\Vert
+\left\Vert x-y\right\Vert ,$

\item[(ii)] $\left\Vert y-Ty\right\Vert \leq 7\left\Vert x-Tx\right\Vert
+2\left\Vert x-y\right\Vert $ \noindent
\end{description}

\noindent hold for all $x,y\in K.$
\end{proposition}

\begin{proof}
(i) is proved in \cite{karapinar}. It follows from (i) that%
\begin{equation*}
\left\Vert y-Ty\right\Vert \leq \left\Vert y-x\right\Vert +\left\Vert
x-Ty\right\Vert \leq 7\left\Vert x-Tx\right\Vert +2\left\Vert x-y\right\Vert
.
\end{equation*}%
This completes the proof.
\end{proof}

\section{Main Results}

In this section, we give a weak and a strong convergence theorem. First, we
prove a couple of lemmas and a proposition which are useful for our weak
convergence theorem. The following lemma is an extension of Lemma 8 of \cite%
{khansuz} to the case of mappings satisfying (RSC)-condition.

\begin{lemma}
\label{1} Let $T$ be a mapping on a bounded and convex subset $K$ of a
uniformly convex Banach space $E.$ Assume that $T$ satisfies $(RSC)$%
-condition. Then for any $\varepsilon >0,$ there exists $\xi \left(
\varepsilon \right) >0$ such that for any $t\in \lbrack 0,1]$ and for any $%
u,v\in K$ with $\left\Vert Tu-u\right\Vert <\xi \left( \varepsilon \right) ,$
$\left\Vert Tv-v\right\Vert <\xi \left( \varepsilon \right) ,$ we have 
\begin{equation*}
\left\Vert T(tu+(1-t)v)-(tu+(1-t)v)\right\Vert <\varepsilon .
\end{equation*}
\end{lemma}

\begin{proof}
Assume on contrary that there exist sequences $\left\{ u_{n}\right\}
,\left\{ v_{n}\right\} \subset K,$ $\left\{ t_{n}\right\} \subset \left[ 0,1%
\right] $ and $\varepsilon >0$ such that $\left\Vert Tu_{n}-u_{n}\right\Vert
<1/n,$ $\left\Vert Tv_{n}-v_{n}\right\Vert <1/n$ \ and 
\begin{equation*}
\left\Vert
T(t_{n}u_{n}+(1-t_{n})v_{n})-(t_{n}u_{n}+(1-t_{n})v_{n})\right\Vert \geq
\varepsilon .
\end{equation*}%
Set $x_{n}=t_{n}u_{n}+(1-t_{n})v_{n}$ and $w_{n}=Tx_{n}.$ From Proposition %
\ref{K} (ii), we get%
\begin{eqnarray*}
0 &<&\varepsilon \leq \liminf_{n\rightarrow \infty }\left\Vert
Tx_{n}-x_{n}\right\Vert \\
&\leq &\liminf_{n\rightarrow \infty }\left( 7\left\Vert
Tu_{n}-u_{n}\right\Vert +2\left\Vert u_{n}-x_{n}\right\Vert \right) \\
&=&2\liminf_{n\rightarrow \infty }\left\Vert u_{n}-x_{n}\right\Vert .
\end{eqnarray*}%
Similarly, we can show that 
\begin{equation*}
0<\liminf_{n\rightarrow \infty }\left\Vert v_{n}-x_{n}\right\Vert
\end{equation*}%
and hence 
\begin{equation*}
0<\liminf_{n\rightarrow \infty }\left\Vert u_{n}-v_{n}\right\Vert .
\end{equation*}%
Since $K$ is bounded and%
\begin{equation*}
0<\liminf_{n\rightarrow \infty }\left\Vert v_{n}-x_{n}\right\Vert
=\liminf_{n\rightarrow \infty }t_{n}\left\Vert u_{n}-v_{n}\right\Vert \leq
\liminf_{n\rightarrow \infty }t_{n}\times \sup_{n\in 
\mathbb{N}
}\left\Vert u_{n}-v_{n}\right\Vert ,
\end{equation*}%
we get $0<\liminf_{n\rightarrow \infty }t_{n}.$ Similarly, we can show that $%
\limsup_{n\rightarrow \infty }t_{n}<1.$ So, without loss of generality, we
may assume that $\left\Vert u_{n}-v_{n}\right\Vert \rightarrow d\in \left(
0,\infty \right) $ and $t_{n}\rightarrow t\in \left( 0,1\right) $ as $%
n\rightarrow \infty .$ From $\lim_{n\rightarrow \infty }\left\Vert
Tu_{n}-u_{n}\right\Vert =0$ and $0<\liminf_{n\rightarrow \infty }\left\Vert
u_{n}-x_{n}\right\Vert ,$ for sufficiently large $n\in 
\mathbb{N}
,$ we obtain%
\begin{equation*}
\frac{1}{2}\left\Vert Tu_{n}-u_{n}\right\Vert \leq \left\Vert
u_{n}-x_{n}\right\Vert .
\end{equation*}%
Since $T$ satisfies (RSC)-condition, we have%
\begin{equation*}
\left\Vert Tu_{n}-Tx_{n}\right\Vert \leq \frac{1}{3}\left( \left\Vert
u_{n}-x_{n}\right\Vert +\left\Vert x_{n}-Tx_{n}\right\Vert +\left\Vert
u_{n}-Tu_{n}\right\Vert \right) .
\end{equation*}%
Similarly, we can show that%
\begin{equation*}
\left\Vert Tv_{n}-Tx_{n}\right\Vert \leq \frac{1}{3}\left( \left\Vert
v_{n}-x_{n}\right\Vert +\left\Vert x_{n}-Tx_{n}\right\Vert +\left\Vert
v_{n}-Tv_{n}\right\Vert \right)
\end{equation*}%
for sufficiently large $n\in 
\mathbb{N}
.$ Now using Propositon \ref{K} (ii) and the definiton of (RSC)-condition,
we get%
\begin{eqnarray*}
\limsup_{n\rightarrow \infty }\left\Vert u_{n}-w_{n}\right\Vert &\leq
&\limsup_{n\rightarrow \infty }\left( \left\Vert u_{n}-Tu_{n}\right\Vert
+\left\Vert Tu_{n}-Tx_{n}\right\Vert \right) \\
&\leq &\limsup_{n\rightarrow \infty }\left( 
\begin{array}{c}
\left\Vert u_{n}-Tu_{n}\right\Vert +\frac{1}{3}\left( \left\Vert
u_{n}-x_{n}\right\Vert \right. \\ 
\left. +\left\Vert x_{n}-Tx_{n}\right\Vert +\left\Vert
u_{n}-Tu_{n}\right\Vert \right)%
\end{array}%
\right) \\
&\leq &\limsup_{n\rightarrow \infty }\left( 
\begin{array}{c}
\left\Vert u_{n}-Tu_{n}\right\Vert +\frac{1}{3}\left( \left\Vert
u_{n}-x_{n}\right\Vert \right. \\ 
\left. +8\left\Vert u_{n}-Tu_{n}\right\Vert +2\left\Vert
u_{n}-x_{n}\right\Vert \right)%
\end{array}%
\right) \\
&=&\left( 1-t\right) d
\end{eqnarray*}%
and%
\begin{eqnarray*}
\limsup_{n\rightarrow \infty }\left\Vert v_{n}-w_{n}\right\Vert &\leq
&\limsup_{n\rightarrow \infty }\left( \left\Vert v_{n}-Tv_{n}\right\Vert
+\left\Vert Tv_{n}-Tx_{n}\right\Vert \right) \\
&\leq &\limsup_{n\rightarrow \infty }\left( 
\begin{array}{c}
\left\Vert v_{n}-Tv_{n}\right\Vert +\frac{1}{3}\left( \left\Vert
v_{n}-x_{n}\right\Vert \right. \\ 
\left. +\left\Vert x_{n}-Tx_{n}\right\Vert +\left\Vert
v_{n}-Tv_{n}\right\Vert \right)%
\end{array}%
\right) \\
&\leq &\limsup_{n\rightarrow \infty }\left( 
\begin{array}{c}
\left\Vert v_{n}-Tv_{n}\right\Vert +\frac{1}{3}\left( \left\Vert
v_{n}-x_{n}\right\Vert \right. \\ 
\left. +8\left\Vert v_{n}-Tv_{n}\right\Vert +2\left\Vert
v_{n}-x_{n}\right\Vert \right)%
\end{array}%
\right) \\
&=&td.
\end{eqnarray*}%
It then follows from Lemma \ref{KS} that%
\begin{equation*}
0<\varepsilon \leq \lim_{n\rightarrow \infty }\left\Vert
x_{n}-w_{n}\right\Vert =0,
\end{equation*}%
which is a contradiction. This completes the proof.
\end{proof}

\begin{proposition}
\label{2}Let $T$ be a mapping on a bounded and convex subset $K$ of a
uniformly convex Banach space $E$. Assume that $T$ satisfies $(RSC)$%
-condition. Then $I-T$ is demiclosed at zero. That is, if $\left\{
x_{n}\right\} $ in $K$ converges weakly to $x_{0}\in K$ and $%
\lim_{n\rightarrow \infty }\left\Vert Tx_{n}-x_{n}\right\Vert =0$ then $%
Tx_{0}=x_{0}$.
\end{proposition}

\begin{proof}
Take a function $\xi :\left( 0,\infty \right) \rightarrow \left( 0,\infty
\right) $ satisfying the conclusion of Lemma \ref{1}. Let $\left\{
x_{n}\right\} $ be a sequence which converges weakly to $x_{0}\in K$ and $%
\lim_{n\rightarrow \infty }\left\Vert Tx_{n}-x_{n}\right\Vert =0.$ Choose $%
\varepsilon >0$ arbitrarily. Define a strictly decreasing sequence $\left\{
\varepsilon _{n}\right\} $ in $(0,\infty )$ by%
\begin{equation*}
\varepsilon _{1}=\varepsilon \text{ and }\varepsilon _{n+1}=\min \left\{
\varepsilon _{n},\xi \left( \varepsilon _{n}\right) \right\} /2.
\end{equation*}%
It is easy to see that $\varepsilon _{n+1}<\xi \left( \varepsilon
_{n}\right) $. Choose a subsequence $\left\{ x_{f\left( n\right) }\right\} $
of $\left\{ x_{n}\right\} $ such that $\left\Vert x_{f\left( n\right)
}-Tx_{f\left( n\right) }\right\Vert <\xi \left( \varepsilon _{n}\right) .$
Note that $x_{0}$ belongs to the closed convex hull of $\left\{ x_{f\left(
n\right) }:n\in 
\mathbb{N}
\right\} $ because it is a weak limit of $\left\{ x_{f\left( n\right)
}\right\} .$ Hence, there exist $y\in K$ and $v\in 
\mathbb{N}
$ such that $\left\Vert y-x_{0}\right\Vert <\varepsilon $ and $y$ belongs to
the convex hull of $\left\{ x_{f\left( n\right) }:n=1,2,\ldots ,v\right\} .$
Using Lemma \ref{1}, we have $\left\Vert Ty-y\right\Vert <\varepsilon .$ Now
Proposition \ref{K} plays it role to assure that%
\begin{equation*}
\left\Vert Tx_{0}-x_{0}\right\Vert \leq 7\left\Vert Ty-y\right\Vert
+2\left\Vert y-x_{0}\right\Vert <9\varepsilon .
\end{equation*}%
Since $\varepsilon >0$ is arbitrary, we obtain $Tx_{0}=x_{0}.$
\end{proof}

\begin{lemma}
\label{3} Let $T$ be a mapping on a bounded and convex subset $K$ of a
uniformly convex Banach space $E$ satisfying $(RSC)$-condition. Let $\left\{
x_{n}\right\} $ be a sequence in $K$ defined by $x_{n+1}=\alpha
Tx_{n}+\left( 1-\alpha \right) x_{n}$, where $x_{1}\in K$ is arbitrariy but
fixed and $\alpha $ is a real number belonging to $[1/2,1)$. Let $p,q\in
F(T) $ and $t\in \left[ 0,1\right] .$ If $\lim_{n\rightarrow \infty
}\left\Vert Tx_{n}-x_{n}\right\Vert =0,$ then $\lim_{n\rightarrow \infty
}\left\Vert tx_{n}+\left( 1-t\right) p-q\right\Vert $ exists.
\end{lemma}

\begin{proof}
Since $T$ satisfies (RSC)-condition, it is quasi-nonexpansive by Proposition %
\ref{krap6}. Let $S$ be a mapping from $K$ onto itself defined by $Sx=\alpha
Tx+\left( 1-\alpha \right) x.$ It is not difficult to see that $F\left(
S\right) =F\left( T\right) $ and $S$ is quasi-nonexpansive. Note that $%
x_{n+1}=\alpha Tx_{n}+\left( 1-\alpha \right) x_{n}=Sx_{n}=S^{n}x_{1}.$ Thus
for any $q\in F\left( S\right) ,$ we have 
\begin{eqnarray*}
\left\Vert x_{n+1}-q\right\Vert &\leq &\left\Vert Sx_{n}-q\right\Vert \\
&\leq &\left\Vert x_{n}-q\right\Vert ,
\end{eqnarray*}%
because $S$ is quasi-nonexpansive. Thus we have%
\begin{equation}
\left\Vert x_{n+1}-q\right\Vert \leq \left\Vert x_{n}-q\right\Vert
\label{15}
\end{equation}%
which means that the sequence $\left\{ \left\Vert x_{n}-q\right\Vert
\right\} $ is nonincreasing and hence converges. Also it is obvious that the
sequence $\left\{ \left\Vert p-q\right\Vert \right\} $ converges. Thus it
suffices to consider $t\in \left( 0,1\right) .$ Let $\lim_{n\rightarrow
\infty }\left\Vert x_{n}-p\right\Vert =d.$ If $d=0,$ there is nothing to
prove. Take $d>0.$ From hypothesis, we have $\lim_{n\rightarrow \infty
}\left\Vert Tx_{n}-x_{n}\right\Vert =0.$ We also have\noindent 
\begin{eqnarray*}
&&\liminf_{m,n\rightarrow \infty }\left\Vert x_{n}-S^{\ell }\left(
tx_{m}+\left( 1-t\right) p\right) \right\Vert \\
&\geq &\liminf_{m,n\rightarrow \infty }\left( \left\Vert x_{n}-p\right\Vert
-\left\Vert p-S^{\ell }\left( tx_{m}+\left( 1-t\right) p\right) \right\Vert
\right) \\
&\geq &\liminf_{m,n\rightarrow \infty }\left( \left\Vert x_{n}-p\right\Vert
-\left\Vert p-\left( tx_{m}+\left( 1-t\right) p\right) \right\Vert \right) \\
&=&\left( 1-t\right) d>0
\end{eqnarray*}%
\newline
for all $\ell \in 
\mathbb{N}
\cup \left\{ 0\right\} $, where $S^{0}$ is the identity mapping on $K$.
Hence, there exists $\nu \in 
\mathbb{N}
$ such that%
\begin{equation*}
\frac{1}{2}\left\Vert x_{n}-Tx_{n}\right\Vert \leq \left\Vert x_{n}-S^{\ell
}\left( tx_{m}+\left( 1-t\right) p\right) \right\Vert
\end{equation*}%
for all $\ell \geq 0$ and $m$, $n\geq \nu $. \ Using (RSC)-condition and
Proposition \ref{K} (ii), we obtain%
\begin{eqnarray*}
\!\left\Vert Tx_{n}-T\circ S^{\ell }\left( tx_{m}+\left( 1-t\right) p\right)
\right\Vert &\leq &\frac{1}{3}\left\Vert x_{n}-S^{\ell }\left( tx_{m}+\left(
1-t\right) p\right) \right\Vert \\
&&+\frac{1}{3}\left\Vert 
\begin{array}{c}
S^{\ell }\left( tx_{m}+\left( 1-t\right) p\right) \\ 
-T\circ S^{\ell }\left( tx_{m}+\left( 1-t\right) p\right)%
\end{array}%
\right\Vert \\
&&+\frac{1}{3}\left\Vert x_{n}-Tx_{n}\right\Vert
\end{eqnarray*}%
\noindent This gives%
\begin{eqnarray*}
&&\left\Vert x_{n+1}-S^{\ell +1}\left( tx_{m}+\left( 1-t\right) p\right)
\right\Vert \\
&\leq &\left\Vert Sx_{n}-S\circ S^{\ell }\left( tx_{m}+\left( 1-t\right)
p\right) \right\Vert \\
&\leq &\left\Vert 
\begin{array}{c}
\alpha Tx_{n}+\left( 1-\alpha \right) x_{n} \\ 
-\alpha T\circ S^{\ell }\left( tx_{m}+\left( 1-t\right) p\right) \\ 
-\left( 1-\alpha \right) S^{\ell }\left( tx_{m}+\left( 1-t\right) p\right)%
\end{array}%
\right\Vert \\
&=&\left\Vert 
\begin{array}{c}
\alpha \left( Tx_{n}-T\circ S^{\ell }\left( tx_{m}+\left( 1-t\right)
p\right) \right) \\ 
+\left( 1-\alpha \right) \left( x_{n}-S^{\ell }\left( tx_{m}+\left(
1-t\right) p\right) \right)%
\end{array}%
\right\Vert \\
&\leq &\alpha \left\Vert Tx_{n}-T\circ S^{\ell }\left( tx_{m}+\left(
1-t\right) p\right) \right\Vert \\
&&+\left( 1-\alpha \right) \left\Vert x_{n}-S^{\ell }\left( tx_{m}+\left(
1-t\right) p\right) \right\Vert
\end{eqnarray*}

\begin{eqnarray}
&\leq &\alpha \left\Vert x_{n}-S^{\ell }\left( tx_{m}+\left( 1-t\right)
p\right) \right\Vert  \notag \\
&&+\frac{8}{3}\left\Vert x_{n}-Tx_{n}\right\Vert  \notag \\
&&+\left( 1-\alpha \right) \left\Vert x_{n}-S^{\ell }\left( tx_{m}+\left(
1-t\right) p\right) \right\Vert  \notag \\
&=&\left\Vert x_{n}-S^{\ell }\left( tx_{m}+\left( 1-t\right) p\right)
\right\Vert +\frac{8}{3}\left\Vert x_{n}-Tx_{n}\right\Vert  \label{xst1}
\end{eqnarray}

\noindent for all $\ell \geq 0$ and $m$, $n\geq \nu $. Let $h:%
\mathbb{N}
\rightarrow \left[ 0,\infty \right) $ be a function defined by%
\begin{equation*}
h\left( n\right) =\left\Vert tx_{n}+\left( 1-t\right) p-q\right\Vert .
\end{equation*}%
Take two subsequences $\left\{ f\left( n\right) \right\} $ and $\left\{
g\left( n\right) \right\} $ of $\left\{ n\right\} $ such that $\nu <f\left(
1\right) ,$ $f\left( n\right) <g\left( n\right) $ for each $n\in 
\mathbb{N}
$ and%
\begin{equation*}
\lim_{n\rightarrow \infty }h\left( f\left( n\right) \right)
=\liminf_{n\rightarrow \infty }h\left( n\right) \text{, }\lim_{n\rightarrow
\infty }h\left( g\left( n\right) \right) =\limsup_{n\rightarrow \infty
}h\left( n\right) .
\end{equation*}%
Set $u_{n}=x_{g\left( n\right) }$, $v_{n}=p$ and $w_{n}=S^{g\left( n\right)
-f\left( n\right) }\left( tx_{f\left( n\right) }+\left( 1-t\right) p\right)
. $ Then we get that $\lim_{n\rightarrow \infty }\left\Vert
u_{n}-v_{n}\right\Vert =d,$%
\begin{eqnarray*}
\limsup_{n\rightarrow \infty }\left\Vert u_{n}-w_{n}\right\Vert
&=&\limsup_{n\rightarrow \infty }\left\Vert x_{g\left( n\right) }-S^{g\left(
n\right) -f\left( n\right) }\left( tx_{f\left( n\right) }+\left( 1-t\right)
p\right) \right\Vert \\
&\leq &\limsup_{n\rightarrow \infty }\left\Vert x_{f\left( n\right) }-\left(
tx_{f\left( n\right) }+\left( 1-t\right) p\right) \right\Vert \\
&&+\frac{8}{3}\limsup_{n\rightarrow \infty }\left\Vert
x_{n}-Tx_{n}\right\Vert \text{ \ }(\text{by }(\ref{xst1})) \\
&=&\left( 1-t\right) \limsup_{n\rightarrow \infty }\left\Vert x_{f\left(
n\right) }-p\right\Vert \\
&=&\left( 1-t\right) d
\end{eqnarray*}%
and $\limsup_{n\rightarrow \infty }\left\Vert v_{n}-w_{n}\right\Vert \leq
td. $ From Lemma \ref{KS}, we get that%
\begin{equation*}
\lim_{n\rightarrow \infty }\left\Vert tx_{g\left( n\right) }+\left(
1-t\right) p-S^{g\left( n\right) -f\left( n\right) }\left( tx_{f\left(
n\right) }+\left( 1-t\right) p\right) \right\Vert =0.
\end{equation*}%
Making use of quasi-nonexpansiveness of $S$ together with the above
equation, we get%
\begin{eqnarray*}
\limsup_{n\rightarrow \infty }h\left( n\right) &=&\lim_{n\rightarrow \infty
}h\left( g\left( n\right) \right) \\
&\leq &\limsup_{n\rightarrow \infty }\left( 
\begin{array}{c}
\left\Vert 
\begin{array}{c}
tx_{g\left( n\right) }+\left( 1-t\right) p \\ 
-S^{g\left( n\right) -f\left( n\right) }\left( tx_{f\left( n\right) }+\left(
1-t\right) p\right)%
\end{array}%
\right\Vert \\ 
+\left\Vert S^{g\left( n\right) -f\left( n\right) }\left( tx_{f\left(
n\right) }+\left( 1-t\right) p\right) -q\right\Vert%
\end{array}%
\right) \\
&=&\limsup_{n\rightarrow \infty }\left\Vert S^{g\left( n\right) -f\left(
n\right) }\left( tx_{f\left( n\right) }+\left( 1-t\right) p\right)
-q\right\Vert \\
&\leq &\limsup_{n\rightarrow \infty }\left\Vert \left( tx_{f\left( n\right)
}+\left( 1-t\right) p\right) -q\right\Vert \\
&\leq &\lim_{n\rightarrow \infty }h\left( f\left( n\right) \right) \\
&=&\liminf_{n\rightarrow \infty }h\left( n\right) .
\end{eqnarray*}%
Thus $\lim_{n\rightarrow \infty }h\left( n\right) $ $=\lim_{n\rightarrow
\infty }\left\Vert tx_{n}+\left( 1-t\right) p-q\right\Vert $ exists.
\end{proof}

Now, we can prove the following weak convergence theorem.

\begin{theorem}
\label{4} Let $E$ be a uniformly convex Banach space whose dual has the
Kadec--Klee property. Under the assumptions of \ Lemma\ \ref{3}\ on $T,K$
and $\left\{ x_{n}\right\} ,\left\{ x_{n}\right\} $ converges weakly to a
fixed point of $T$ .
\end{theorem}

\begin{proof}
Let $W$ be the set of all weak subsequential limits of $\left\{
x_{n}\right\} .$ Since $\lim_{n\rightarrow \infty }\left\Vert
Tx_{n}-x_{n}\right\Vert =0,$ by Proposition \ref{2} we have $W\subset
F\left( T\right) .$ That $W$ is singleton now follows by using Lemma \ref{FK}
and Lemma \ref{3}. Since $E$ is reflexive (being uniformly convex), every
subsequence of $\left\{ x_{n}\right\} $ has a subsequence converging weakly
to the unique element of $W$. Therefore $\left\{ x_{n}\right\} $ itself
converges weakly to the unique element of $W$.
\end{proof}

\begin{remark}
The above theorem generalizes Theorem $11$ of Khan and Suzuki \cite{khansuz}.
\end{remark}

The following is also a direct consequence of Theorem \ref{4}.

\begin{corollary}
Let $E$ be a uniformly convex Banach space whose norm is Fr\'{e}chet
differentiable. Let $K,T,$and $\left\{ x_{n}\right\} $ be as in Lemma\ \ref%
{3}. Then, $\left\{ x_{n}\right\} $ converges weakly to a fixed point of $T.$
\end{corollary}

Now we prove strong convergence theorem by using Condition (I) of Sentor and
Dotson \cite{Senter}.

Recall that a mapping $T:K\rightarrow K$ where $K$ is a subset of $E$, is
said to satisfy condition (I) \cite{Senter} if there exists a nondecreasing
function $f:[0,\infty )\rightarrow \lbrack 0,\infty )$ with $f(0)=0$, $%
f(r)>0 $ for all $r\in (0,\infty )$ such that $\left\Vert x-Tx\right\Vert
\geq f\left( d\left( x,F\left( T\right) \right) \right) $ for all $x\in K$
where $d\left( x,F\left( T\right) \right) =\inf \left\{ \left\Vert
x-p\right\Vert :p\in F\left( T\right) \right\} .$

\begin{theorem}
Let $E$ be a uniformly convex Banach space. Let $K$, $T$, and $\left\{
x_{n}\right\} $ be as in Lemma\ $\ref{3}$. If $T$ satisfies the condition $%
(I),$ then $\left\{ x_{n}\right\} $ converges strongly to a fixed point of $%
T $ .
\end{theorem}

\begin{proof}
By Lemma \ref{3}, we know that $\lim_{n\rightarrow \infty }\left\Vert
x_{n}-p\right\Vert $ exists for all $p\in F\left( T\right) $. From the
inequality (\ref{15}),%
\begin{equation*}
d\left( x_{n+1},F\left( T\right) \right) \leq d\left( x_{n},F\left( T\right)
\right)
\end{equation*}%
$\ $and so $\lim_{n\rightarrow \infty }d\left( x_{n},F\left( T\right)
\right) $ exists. Assume that $\lim_{n\rightarrow \infty }\left\Vert
x_{n}-p\right\Vert =c$ for some $c\geq 0.$ If $c=0$, it is clear that $%
\left\{ x_{n}\right\} $ converges strongly to $p.$ Suppose $c>0.$ From
hypothesis and the condition (I), we have $\lim_{n\rightarrow \infty
}\left\Vert Tx_{n}-x_{n}\right\Vert =0$ and $f\left( d\left( x_{n},F\left(
T\right) \right) \right) \leq \left\Vert Tx_{n}-x_{n}\right\Vert $. This
gives $\lim_{n\rightarrow \infty }f\left( d\left( x_{n},F\left( T\right)
\right) \right) =0.$ Since $f$ is nondecreasing function, we have $%
\lim_{n\rightarrow \infty }d\left( x_{n},F\left( T\right) \right) =0.$ Thus,
we can take a subsequence $\left\{ x_{n_{k}}\right\} $ of $\left\{
x_{n}\right\} $ and a sequence $\left\{ y_{k}\right\} \subset F\left(
T\right) $ such that $\left\Vert x_{n_{k}}-y_{k}\right\Vert <2^{-k}.$ So, it
follows from method of proof of Tan and Xu \cite{tanxu} that $\left\{
y_{k}\right\} $ is a Cauchy sequence in $F(T)$ and so it converges to a
point $y.$ Since $F\left( T\right) $ is closed, therefore $y\in F\left(
T\right) $ and then $\left\{ x_{n_{k}}\right\} $ converges strongly to $y.$
Since $\lim_{n\rightarrow \infty }\left\Vert x_{n}-p\right\Vert $ exists, we
have that $x_{n}\rightarrow y\in F\left( T\right) .$ This completes the
proof.
\end{proof}

\end{document}